\documentclass{amsart}
\usepackage{amsmath,amssymb,amsthm,amscd}
\newtheorem{formula}{}[section]
\newtheorem{proposition}[formula]{Proposition}
\newtheorem{corollary}[formula]{Corollary}
\newtheorem{lemma}[formula]{Lemma}
\newtheorem{theorem}[formula]{Theorem}
\theoremstyle{definition}
\newtheorem{definition}[formula]{Definition}
\newtheorem{example}[formula]{Example}
\theoremstyle{remark}
\newtheorem*{remark}{Remark}

\begin{document}
\title[Complex structures on nilpotent Lie algebras]{Complex structures on nilpotent Lie algebras
and descending central series}
\author[Dmitri V. Millionshchikov]{Dmitri V. Millionshchikov}
\thanks{Partially supported by
the Russian Foundation for Fundamental Research, grant no. 14-01-00537.}
\subjclass{Primary 17B30; Secondary 32G05, 53C30}
\keywords{Invariant complex structure, nilpotent Lie algebra, nilmanifold,
nil-index}
\address{Department of Mathematics and Mechanics, Moscow
State University, 119992 Moscow, RUSSIA}
\email{million@mech.math.msu.su}

\begin{abstract}
We study the algebraic constraints on the structure of nilpotent Lie algebra $\mathbb{g}$, 
which arise because of the presence of an integrable complex structure $J$. Particular attention is paid to non-abelian complex structures. Constructed various examples of positive graded Lie algebras with complex structures, in particular, we construct an infinite family $\mathfrak{D}(n)$ of such algebras that we have for their nil-index $s(\mathfrak{D}(n))$:
$$
s(\mathfrak{D}(n))=\left[ \frac{2}{3}\dim{\mathfrak{D}(n)} \right].
$$
\end{abstract}

\date{}

\maketitle

\section*{Introduction}

The Newlander-Nirenberg theorem \cite{NN} implies that a
left-invariant complex structure on a real simply connected Lie
group $G$ can be defined as an almost-complex structure $J$ on the
tangent Lie algebra $\mathfrak{g}$ of $G$ ($J$ is a linear
endomorphism of $\mathfrak{g}$ such that $J^2=-1$) satisfying {\it
the integrability condition} (\ref{Nijenhuis}) (vanishing of the
Nijenhuis tensor):
$$
[JX,JY]=[X,Y]+J[JX,Y]+J[X,JY], \; \forall X,Y \in \mathfrak{g}.
$$
Extending an almost complex structure $J$ on the complexification
$\mathfrak{g}^{\mathbb C}$ it is easy to see that the
integrability condition is equivalent to the following one:
the eigen-spaces $\mathfrak{g}^{\mathbb C}_{\pm i}$ of $J$
corresponding to the eigen-values $\pm i$ are
(complex) subalgebras of $\mathfrak{g}^{\mathbb C}$.
If they are abelian subalgebras then the complex structure $J$
is called {\it abelian}. It was proved in \cite{DF2} that a Lie group
$G$ admitting a left-invariant abelian complex structure has to be
solvable. On the another hand an abelian complex structure
is nilpotent \cite{Sal}. The study of nilpotent complex
structures on nilmanifolds (nilpotent Lie algebras)
was the subject in \cite{CFGU}.
The properties of nilmanifolds with abelian complex structures is much more studied than the general case. For instance, the Dolbeault cohomology of a nilmanifold with an integrable abelian complex structure can be expressed in terms of the corresponding Lie algebra cohomology (\cite{CF}, \cite{CFGU}).

Existing finite list of all real $6$-dimensional nilpotent algebras up to isomorphism \cite{Moros} allowed S.Salamon in \cite{Sal} distinguish among them algebras admitting integrable complex structures, spend their classification from this point of view. In his cassification 
there are examples of  Lie algebras
that admit only non-abelian complex structures (as well as
examples of Lie algebras that does not admit any complex
structure). This approach does not work in the following even dimension $8$, where, on the one hand, such a classification does not exist, on the other - there are infinite families of pairwise nonisomorphic nilpotent algebras.

Another way is to find a priori algebraic constraints expressed in terms of the nil-index, the dimensions of the ideals of descending central series, the first Betti number, etc., which narrow the range of possible candidates to possess an integrable complex (or hypercomplex) structure.
This approach has been implemented in \cite{DF1}, \cite{DF2}, where the authors managed to classify $8$-dimensional real nilpotent algebras admitting hypercomplex structure, despite the lack of a general classification.

As an example of an algebraic constraints, which we discussed above, we can cite the following general result \cite{GzR} that a filiform Lie algebra $\mathfrak{g}$ (i.e. a nilpotent Lie algebra with 
the value of nilindex $s(\mathfrak{g})=\dim{\mathfrak{g}}-1$)
does not admit any integrable complex structure. 
In this article we prove the estimate (\ref{no_filiform})
$$
\dim{\mathfrak{g}} - \dim{\mathfrak{g}^4}=\dim{\mathfrak{g}} - \dim{\left[\mathfrak{g}, \left[\mathfrak{g}, [\mathfrak{g},\mathfrak{g}]\right]\right]}
\ge 5.
$$
From this estimate follows easily \cite{GzR}.

The main purpose of this artcile is to study nilpotent algebras admitting (non-abelian) complex structures in high (arbitrary) dimensions. This requires a stock of examples of such algebras. As suitable examples we propose to study positively graded Lie algebras. On the one hand it is more easy to study them (cohomology computations) on the other hand they have quite interesting properties. In addition, any nilpotent algebra can be obtained as a special deformation of a positively graded Lie algebra.

The Theorem \ref{main_theorem} claims that for the maximal value $s(2n)$ of nil-index $s(\mathfrak{g})$ of $2n$-dimensional nilpotent Lie algebras 
$\mathfrak{g}$ admiting a complex structure  we have the following estimates (\ref{main_estimate}):
$$
\left[\frac{4n}{3} \right] \le s(2n) \le 2n-2.
$$
It follows from \cite{Sal} that $s(6)=4$. It appears possible to prove that $s(8)=5$  and improve the estimates (\ref{main_estimate}) for higher dimensions.

\section{Nilpotent Lie algebras}

The sequence of ideals of a Lie algebra $\mathfrak{g}$
$$\mathfrak{g}^1=\mathfrak{g} \; \supset \;
\mathfrak{g}^2=[\mathfrak{g},\mathfrak{g}] \; \supset \; \dots
\; \supset \;
\mathfrak{g}^k=[\mathfrak{g}, \mathfrak{g}^{k-1}] \; \supset
\; \dots$$
is called the descending central sequence of $\mathfrak{g}$.

A Lie algebra $\mathfrak{g}$ is called nilpotent if
there exists a natural number $s(\mathfrak{g})$ such that:
$$\mathfrak{g}^{s(\mathfrak{g})+1}=[\mathfrak{g},
\mathfrak{g}^{s(\mathfrak{g})}]=0,
\quad \mathfrak{g}^{s(\mathfrak{g})} \: \ne 0.$$
$s(\mathfrak{g})$ is called the
nil-index of the nilpotent Lie algebra $\mathfrak{g}$ and
$\mathfrak{g}$ is called $s(\mathfrak{g})$-step nilpotent Lie algebra.
Thus one can regard an abelian Lie algebra as $1$-step
nilpotent.
\begin{example}
\label{heisenberg}
The Heisenberg algebra $\mathfrak{h}_{2k{+}1}$ is defined
by its basis $x_1, y_1, \dots, x_k, y_k, z$
and the commutating relations:
$$
[x_i,y_i]=z, \; i=1, \dots, k.
$$
\begin{remark}
In the sequel we will omit trivial relations $[e_i,e_j]=0$
in the definitions of Lie algebras.
\end{remark}
The Heisenberg Lie algebra $\mathfrak{h}_{2k{+}1}$
is $2$-step nilpotent.
\end{example}

One can consider the sequence
$$a(\mathfrak{g})=\left( a_1(\mathfrak{g}), \dots,
a_{s(\mathfrak{g})}(\mathfrak{g}) \right),$$
where
$a_i(\mathfrak{g})=\dim{\mathfrak{g}^i/\mathfrak{g}^{i{+}1}}$.

We have the following obvious estimates on $a_i(\mathfrak{g})$:
$$
a_1(\mathfrak{g})=\dim (\mathfrak{g}/[\mathfrak{g},\mathfrak{g}])
\: \ge \: 2, \; \; \;
a_i(\mathfrak{g}) \ge \: 1, \; i=2,\dots, s(\mathfrak{g}).
$$

It immediately follows, that we have the following estimate:
$$s(\mathfrak{g}) \le \dim{\mathfrak{g}}-1.$$

\begin{definition}
A finite dimensional nilpotent Lie algebra $\mathfrak{g}$ is  called
filiform if $s(\mathfrak{g}) = \dim{\mathfrak{g}}-1$.
\end{definition}
One can remark that in other words a Lie
algebra $\mathfrak{g}$ is filiform iff
$$a(\mathfrak{g})=(2,1,1,\dots,1).$$

\begin{example}
The Lie algebra $\mathfrak{m}_0(n)$ that is defined
by its basis $e_1, e_2, \dots, e_n$
and the commutating relations:
$$ [e_1,e_i]=e_{i+1}, \; i=2, \dots, n-1$$
is obviously an example of a filiform Lie algebra.
\end{example}

Let us consider the cochain complex of a Lie algebra $\mathfrak{g}$ with
$\dim \mathfrak{g}=n$:
$$
\begin{CD}
\mathbb K @>{d_0{=}0}>> \mathfrak{g}^* @>{d_1}>> \Lambda^2 (\mathfrak{g}^*) @>{d_2}>>
\dots @>{d_{n-1}}>>\Lambda^{n} (\mathfrak{g}^*) @>>> 0
\end{CD}
$$

where $d_1: \mathfrak{g}^* \rightarrow \Lambda^2 (\mathfrak{g}^*)$
is the dual mapping to the Lie bracket
$[ \, , ]: \Lambda^2 \mathfrak{g} \to \mathfrak{g}$,
and the differential $d$ (that is a collection of $d_p$)
is a derivation of the exterior algebra $\Lambda^*(\mathfrak{g}^*)$
that continues $d_1$:

$$
d(\rho \wedge \eta)=d\rho \wedge \eta+(-1)^{deg\rho} \rho \wedge d\eta,
\; \forall \rho, \eta \in \Lambda^{*} (\mathfrak{g}^*).
$$
The condition $d^2=0$ is equivalent to the Jacobi identity in $\mathfrak{g}$.

The cohomology of $(\Lambda^{*} (\mathfrak{g}^*), d)$ is
called the cohomology (with trivial coefficients) of the Lie algebra
$\mathfrak{g}$ and is denoted by $H^*(\mathfrak{g})$. One can easily remark that
$H^1(\mathfrak{g})$ is
isomorphic to $\left( \mathfrak{g}/[\mathfrak{g}, \mathfrak{g}] \right)^*$.

Let us define a family $\left\{V_l\mathfrak{g}^*\right\}$ of subspaces in $\mathfrak{g}^*$:

1) $V_0\mathfrak{g}^*=\{0\}$, 

2) $V_1\mathfrak{g}^*={\rm Ker} d_1$,
$d_1f(X,Y)=f([X,Y])$,

3) $V_l\mathfrak{g}^*=\{f \in \mathfrak{g}^*: d_1f \in
\Lambda^2(V_{l{-}1}\mathfrak{g}^*)\}, \; l \ge 2$.
$$
\{0\} \subset V_1\mathfrak{g}^* \subset \dots \subset V_l\mathfrak{g}^* \subset V_{l{+}1}\mathfrak{g}^* \subset
\dots
$$

The first subspace $V_1\mathfrak{g}^*$ is the annihilator
of $\mathfrak{g}^2{=}[\mathfrak{g},\mathfrak{g}]$ and  it is
isomorphic to the first cohomology $H^1(\mathfrak{g})$.
Supposing by induction that $V_{l{-}1}\mathfrak{g}^*$
annihilates $\mathfrak{g}^{l}$
one can remark that $d_1f \in \Lambda^2(V_{l}\mathfrak{g}^*)$ iff
$d_1f(X,Y)=f([X,Y])$ vanishes for all
$X \in \mathfrak{g}$ and $Y \in \mathfrak{g}^{l}$
($f$ annihilates the subspace $\mathfrak{g}^l$).
Hence $V_l\mathfrak{g}^*$ is the annihilator of $\mathfrak{g}^{l{+}1}$.
Also we have
$$a_l(\mathfrak{g})=\dim{\mathfrak{g}^l/\mathfrak{g}^{l{+}1}}=
\dim{V_{l{+}1}\mathfrak{g}^*/V_{l}\mathfrak{g}^*}.$$
Now the nilpotency
condition for a Lie algebra $\mathfrak{g}$ can be interpreted
in a following way:
a non-abelian $\mathfrak{g}$ is $s$-step
nilpotent iff there exists a positive integer $s$ such that $V_{s}\mathfrak{g}^*=\mathfrak{g}^*$ and
$V_{s{-}1}\mathfrak{g}^* \neq \mathfrak{g}^*$.

\section{Integrability condition}

\begin{definition}
An almost-complex structure $J$ on a Lie
algebra $\mathfrak{g}$ (i.e. $J$ is a linear endomorphism
of $\mathfrak{g}$ such that $J^2=-1$)
satisfying the integrability condition
\begin{equation}
\label{Nijenhuis}
[JX,JY]=[X,Y]+J[JX,Y]+J[X,JY], \; \forall X,Y \in \mathfrak{g}
\end{equation}
is called a complex structure on $\mathfrak{g}$.
\end{definition}

\begin{example}
Let us consider the direct sum $\mathfrak{h}_{2k{+}1}\oplus {\mathbb R}$,
where one-dimensional abelian ${\mathbb R}$ is spanned by $w$.
One can define an operator
$J$ on $\mathfrak{h}_{2k{+}1}\oplus {\mathbb R}$:
$$Jx_i=y_i, Jy_i=-x_i \; i=1, \dots, k; \;\;Jz=w,\; Jw=-z.$$
$J^2=-1$ and $J$ satisfies the Nijenhuis condition.
In fact $J$ satisfies to an identity even stronger than
(\ref{Nijenhuis}).
\end{example}

\begin{definition}
An almost complex structure $J$ on a Lie algebra $\mathfrak{g}$
is said to be an abelian complex structure iff
\begin{equation}
\label{abelian}
[JX,JY]=[X,Y], \; \forall X,Y \in \mathfrak{g}.
\end{equation}
\end{definition}
Obviously if an almost complex structure satisfies
(\ref{abelian}) it satisfies the Nijenhuis condition (\ref{Nijenhuis}).
$J$ from Example \ref{heisenberg}
is an abelian complex structure. It was proved in \cite{DF2} that
a real Lie algebra admitting an abelian complex structure
has to be solvable.

Extending an almost complex structure $J$ on the complexification
$\mathfrak{g}^{\mathbb C}$ we have a splitting
$$\mathfrak{g}^{\mathbb C}=\mathfrak{g}^{\mathbb C}_{{-}i}
\oplus \mathfrak{g}^{\mathbb C}_{i},$$
where $\mathfrak{g}^{\mathbb C}_{{\pm }i}=\{x-\pm iJx: x \in \mathfrak{g}\}$
are the eigen-space of the complexification of $J$ corresponding to the
eigen-values $\pm i$. It is easy to see that:

1) $J$ is integrable iff
both $\mathfrak{g}^{\mathbb C}_{{\pm }i}$ are (complex) subalgebras
of $\mathfrak{g}^{\mathbb C}$;

2) $J$ is an abelian complex structure iff
$\mathfrak{g}^{\mathbb C}_{{\pm }i}$ are abelian subalgebras
of $\mathfrak{g}^{\mathbb C}$.

One can point out another one important
case:

3) the eigen-spaces $\mathfrak{g}^{\mathbb C}_{{\pm }i}$ of $J$ are
ideals of $\mathfrak{g}^{\mathbb C}$.

The last condition is equivalent to the following one:
\begin{equation}
\label{complexLie}
[JX,Y]=J[X,Y], \; \forall X,Y \in \mathfrak{g}.
\end{equation}
And it is the definition of a complex Lie algebra structure,
i.e. $(\mathfrak{g}, J)$ can be regarded as a complex Lie algebra.

\begin{example}
\label{compl_filif}
Let us consider a Lie algebra $\mathfrak{m}_{0}(n)^{\mathbb R}, n \ge 2$
defined by its
basis $x_1, y_1, x_2, y_2, \dots, x_n, y_n$ and the
structure relations:
$$
[x_1,x_i]=[y_i,y_1]=x_{i{+}1}, \; [x_1,y_i]=[y_1,x_i]=y_{i{+}1},
\;\;
i=2, \dots, n{-}1.
$$
$2n$-dimensional Lie algebra $\mathfrak{m}_{0}(n)^{\mathbb R}$ is
$(n{-}1)$-step nilpotent.

An almost complex structure
$J$ on $\mathfrak{m}_{0}(n)^{\mathbb R}$ that is defined by
$Jy_i=x_i \; i=1, \dots, n$
satisfies (\ref{complexLie}) and
$(\mathfrak{m}_{0}(n)^{\mathbb R}, J)$
is isomorphic to the complex filiform Lie algebra
$\mathfrak{m}_{0}(n)$.
\end{example}

Now we are going to start study of complex structures on nilpotent
Lie algebras. Let $\mathfrak{g}$ a nilpotent Lie algebra
with integrable complex structure $J$ and
$\{\mathfrak{g}^j\}$ its descending central sequence.

An ideal $\mathfrak{g}^l$ is not in general an invariant subspace
with respect to $J$.
One can consider $\mathfrak{g}^l(J)=\mathfrak{g}^l+J\mathfrak{g}^l$
-- the smallest $J$-invariant subspace of $\mathfrak{g}$
containing $\mathfrak{g}^l$.
We have a decreasing sequence of $J$-invariant subspaces
$$\mathfrak{g}^1(J)=\mathfrak{g} \; \supset \;
\mathfrak{g}^2(J)=[\mathfrak{g},\mathfrak{g}]+J[\mathfrak{g},\mathfrak{g}]
\; \supset \; \dots
\; \supset \;
\mathfrak{g}^{s(\mathfrak{g})}(J) \; \supset \{0\}.$$

\begin{proposition}[S.Salamon, \cite{Sal}]

\label{l-commut}
$$[\mathfrak{g}^l(J), \mathfrak{g}^l(J)] \subset \mathfrak{g}^{l{+}1}(J).$$
\end{proposition}
\begin{proof}
One can take arbitrary $X_1,X_2,Y_1,Y_2 \in \mathfrak{g}^l$. Then
$$
[X_1+JY_1,X_2+JY_2]=[X_1,X_2]+[JY_1,X_2]+[X_1,JY_2]+[JY_1,JY_2].
$$
The first three summands on the right part of the equality
are obviously in $\mathfrak{g}^{l{+}1}$.
It follows from the integrability condition
(\ref{Nijenhuis}) that $[JY_1,JY_2]$ is in $\mathfrak{g}^l+J\mathfrak{g}^l$
as  $[Y_1,Y_2], [JY_1,Y_2], [Y_1,JY_2] \in \mathfrak{g}^{l{+}1}$.
\end{proof}
\begin{corollary}
A subspace $\mathfrak{g}^l(J)$  is a subalgebra in $\mathfrak{g}(J)$
and an ideal in $\mathfrak{g}^{l{-}1}(J)$ for all $l$.
\end{corollary}

\begin{proposition}[S. Salamon, \cite{Sal}]
\label{nontrivial_H}
$$
\mathfrak{g}^2(J)=[\mathfrak{g},\mathfrak{g}]+J[\mathfrak{g},\mathfrak{g}]
\ne \mathfrak{g}=\mathfrak{g}^1(J).
$$
\end{proposition}
\begin{proof}
Let assume that $\mathfrak{g}^2(J)=\mathfrak{g}$, then
exists $2\le j_0 \le s(\mathfrak{g})$
such that
$$\mathfrak{g}^{j_0}(J)= \mathfrak{g}, \quad
\mathfrak{g}^{j_0{+}1}(J)\neq \mathfrak{g}.$$
It follows that $[\mathfrak{g},\mathfrak{g}]=
[\mathfrak{g}^{j_0}(J),\mathfrak{g}^{j_0}(J)] \subset
\mathfrak{g}^{j_0{+}1}(J) \neq \mathfrak{g}$,
the subspace $\mathfrak{g}^{j_0{+}1}(J)$ is $J$-invariant,
hence $J[\mathfrak{g},\mathfrak{g}] \subset \mathfrak{g}^{j_0{+}1}(J)$ also.
Combining these results we have an inclusion
$$\mathfrak{g}^{2}(J)=[\mathfrak{g},\mathfrak{g}]+J[\mathfrak{g},\mathfrak{g}]
\subset \mathfrak{g}^{j_0{+}1}(J) \neq \mathfrak{g}$$
that contradicts to our initial assumption.
\end{proof}

\section{Minimal models and complex structures}
\label{minimalmodels}
Given a Lie algebra $\mathfrak{g}$ with an integrable
complex structure $J$.
One can consider its conjugate $J$ (we will keep the same notation for it)
acting on $\mathfrak{g}^*$. 
$$
Jf(X)=f(JX), f \in \mathfrak{g}^*, X \in \mathfrak{g}.
$$
Proceeding to the complexification of $J$
we have a splitting
\begin{equation}
\label{1,0-split}
(\mathfrak{g}^*)^{\mathbb C}=
\Lambda^{1,0} \oplus \Lambda^{0,1},
\end{equation}
where $\Lambda^{1,0}=\{f-iJf: f \in \mathfrak{g}^* \}$ and
$\Lambda^{0,1}=\{f+iJf: f \in \mathfrak{g}^* \}$ are
the eigen-spaces of the complexification of $J$ that correspond to the
eigen-values $\pm i$ respectively.  Also we have
$\Lambda^{1,0}=(\mathfrak{g}^{\mathbb C}_{i})^*$
and $\Lambda^{0,1}=(\mathfrak{g}^{\mathbb C}_{{-}i})^*$.

The splitting (\ref{1,0-split}) induces a decomposition
$$
\Lambda^k((\mathfrak{g}^{\mathbb C})^*)=
\bigoplus_{p{+}q{=}k}\Lambda^{p,q},
$$
where
$\Lambda^{p,q}=\Lambda^p((\mathfrak{g}^{\mathbb C}_{i})^*)\otimes
\Lambda^q((\mathfrak{g}^{\mathbb C}_{{-}i})^*)$ is the subspace of
$(p,q)$-forms relative to $J$.

For a given subspace $\mathfrak{a} \subset \mathfrak{g}$ let us denote by
$\mathfrak{a}^{ann}$ its annihilator in $\mathfrak{g}^*$:
$$
\mathfrak{a}^{ann}=\left\{ f \in \mathfrak{g}^* | f (X)=0, \forall X \in \mathfrak{a}\right\}.
$$
Now one can consider an obvious lemma:
\begin{lemma}
\label{d-annihil}
Let $\mathfrak{a}$ and $\mathfrak{b}$ be two subspaces of $\mathfrak{g}$ then
$$d \mathfrak{b}^{ann} \subset \mathfrak{a}^{ann} \wedge \mathfrak{g}^*$$
if and only if $[\mathfrak{a},\mathfrak{a}] \subset \mathfrak{b}$.
\end{lemma}

\begin{proof}
Let $ f \in \mathfrak{b}^{ann}$. Then $df \in \mathfrak{a}^{ann} \wedge \mathfrak{g}^*$, or
$$
df(X,Y)=f([X,Y])=0, \forall X,Y \in \mathfrak{a},
$$ 
if and only if,
$[X,Y] \in \mathfrak{b}, \forall X,Y \in \mathfrak{a}$.
\end{proof}

\begin{corollary}
\label{d-annihilator} Let $\mathfrak{a}$ be a subspace of a Lie
algebra $\mathfrak{g}$. The ideal $I({\mathfrak{a}}^{ann})$
generated by the annihilator ${\mathfrak{a}}^{ann}$ in the
exterior algebra $\Lambda(\mathfrak{g}^*)$ is $d$-closed
$$d {\mathfrak{a}}^{ann} \subset {\mathfrak{a}}^{ann}
\wedge \mathfrak{g}^*$$ if and only if
$\mathfrak{a}$ is a subalgebra of $\mathfrak{g}$.
\end{corollary}

Now we can interprete $\Lambda^{1,0}$ as the annihilator
of the subalgebra $\mathfrak{g}^{\mathbb C}_{{-}i}$ and
applying Corollary \ref{d-annihilator} rewrite again
the integrability conditions:

1)  (\ref{Nijenhuis})
holds for an almost complex structure $J$ iff
$$d\Lambda^{1,0} \subset \Lambda^{2,0}\oplus \Lambda^{1,1};$$

2) the abelian property (\ref{abelian}) holds for $J$ iff
$$d\Lambda^{1,0} \subset \Lambda^{1,1};$$

3) $J$ is a complex Lie algebra structure (\ref{complexLie})
iff
$$ d\Lambda^{1,0} \subset \Lambda^{2,0}.$$

Let us consider an increasing sequence of complex subspaces in $\Lambda^{1,0}$:
$$V_0^{1,0}=\{0\} \subset V_1^{1,0} \subset \dots \subset
V_{s(\mathfrak{g})}^{1,0}=\Lambda^{1{,}0},$$
where
$$
V_l^{1,0}=(V_l\mathfrak{g}^*)^{\mathbb C}
\cap \Lambda^{1{,}0}, \; l=0,1,\dots,s(\mathfrak{g}).
$$

\begin{remark}
$V_1^{1,0}$ is the subspace of closed holomorphic $1$-forms.
\end{remark}

\begin{proposition}
$V_{l}^{1,0}$ is the annihilator of
$\tilde {\mathfrak g}^{l{+}1}=\mathfrak{g}^{l{+}1}(J)^{\mathbb C}+\mathfrak{g}^{\mathbb C}_{{-}i}$.
\end{proposition}

\begin{proof}
$$
\left((\mathfrak{g}^{l{+}1})^{\mathbb C}
+\mathfrak{g}^{\mathbb C}_{{-}i}\right)^{ann}=
(V_l\mathfrak{g}^*)^{\mathbb C}
\cap (\mathfrak{g}^{\mathbb C}_{-i})^{ann}=V_l^{1,0}.
$$
But in the same time
$$
\mathfrak{g}^{l{+}1}(J)^{\mathbb C}+\mathfrak{g}^{\mathbb C}_{{-}i}=
({\mathfrak{g}^{l{+}1}})^{\mathbb C}+\mathfrak{g}^{\mathbb C}_{{-}i},
$$
because $X+iJX \in 
\mathfrak{g}^{\mathbb C}_{{-}i}, \;\forall X \in (\mathfrak{g}^{l{+}1})^{\mathbb C}$.
\end{proof}

\begin{proposition}
$\tilde {\mathfrak g}^{l}=(\mathfrak{g}^{l})^{\mathbb C}{+}\mathfrak{g}^{\mathbb C}_{{-}i}{=}
\mathfrak{g}^{l}(J)^{\mathbb
C}{+}\mathfrak{g}^{\mathbb C}_{{-}i}$ is a subalgebra of
$\mathfrak{g}^{\mathbb C}$, moreover
$$[\tilde {\mathfrak g}^{l}, \tilde {\mathfrak g}^{l}] \subset
\tilde {\mathfrak g}^{l{+}1}.$$
\end{proposition}

Applying Lemma \ref{d-annihil} and the previous two propositions we obtain
\begin{corollary}[\cite{Sal}]
$$
d V_l^{1,0}  \subset 
V_{l{-}1}^{1,0}\wedge (\mathfrak{g}^{\mathbb C})^* .
$$
\end{corollary}
However we need to precise this statement 
\begin{lemma}
\label{d-J-model}
\begin{equation}
\label{d-J-property}
d V_l^{1,0} \subset V_{l{-}1}^{1,0}\wedge (V_{l{-}1}\mathfrak{g}^*)^{\mathbb C}.
\end{equation}
\end{lemma}

\begin{proof}
$V_l^{1,0}=(V_l\mathfrak{g}^*)^{\mathbb C}
\cap \Lambda^{1{,}0}$ is a subspace in $(V_l\mathfrak{g}^*)^{\mathbb C}$ and hence
$$
d V_l^{1,0} \subset d(V_l\mathfrak{g}^*)^{\mathbb C} \subset  (V_{l{-}1}\mathfrak{g}^*)^{\mathbb C} \wedge (V_{l{-}1}\mathfrak{g}^*)^{\mathbb C}.
$$
The intersection of two subspaces $V_{l{-}1}^{1,0}\wedge (\mathfrak{g}^{\mathbb C})^*$ and
$(V_{l{-}1}\mathfrak{g}^*)^{\mathbb C} \wedge (V_{l{-}1}\mathfrak{g}^*)^{\mathbb C}$ gives the answer.
\end{proof}

\begin{theorem}[\cite{Sal}]
A real nilpotent  $2n$-dimensional Lie algebra $\mathfrak{g}$ admits
an integrable complex structure if and only if
$(\mathfrak{g}^{\mathbb C})^*$ has a basis
$\{\omega^1, \dots \omega^n, \bar \omega^1, \dots \bar \omega^n \}$
such that
$$d\omega^{l+1} \in I(\omega^1, \dots \omega^l), \; l=0, \dots, n{-}1,$$
where
$I(\omega^1, \dots \omega^l)$ is an ideal in
$\Lambda^*((\mathfrak{g}^{\mathbb C})^*)$ generated by
$\omega^1, \dots \omega^l$.
\end{theorem}

\begin{corollary}[\cite{Sal}]
A nilpotent Lie algebra $\mathfrak{g}$ admits
an abelian complex structure if and only if
$(\mathfrak{g}^*)^{\mathbb C}$ has a basis
$\{\omega^1, \dots \omega^n, \bar \omega^1, \dots \bar \omega^n \}$
such that
$$d\omega^{i+1} \in \Lambda^{1,1}(\omega^1, \dots \omega^i,
\bar \omega^1, \dots \bar \omega^i), \; i=0, \dots, n{-}1,$$
\end{corollary}

\begin{example}
\label{abelian_long}
Define a graded Lie algebra
$\mathfrak{B}(n)=\oplus_{l=1}^{n} \mathfrak{B}_{l}$ such that:
$$\dim \mathfrak{B}_{l}= 2, \; l=1,2,\dots,n.$$
Let $x_{l}, y_{l}$ denote basic elements in
$\mathfrak{B}_{l}$ for $l=1,\dots, n$. Then
a Lie algebra structure of $\mathfrak{B}(n)$ is defined by:
$$
[x_1,y_1]=y_2, \;[x_1,x_l]=[y_1,y_l]=x_{l{+}1},
\; [x_1,y_l]=[x_l,y_1]=y_{l{+}1},
\;\;
l=2,3, \dots, n{-}1.
$$
One can define an abelian
complex structure $J$:
$$Jx_l=-y_l, \;Jy_{l}=x_l, \; l=1,2,\dots,n.$$

Taking $1$-forms $\omega^{l}=x^l+iy^{l}, \; l=1,2,\dots$
we have
\begin{equation}
\begin{split}
d\omega^1=0, \; d\omega^2=\frac{1}{2}\bar \omega^1 {\wedge} \omega^1, 
d\omega^3=\bar \omega^1 {\wedge} \omega^2, \dots,
d\omega^{n}=\bar \omega^1 {\wedge} \omega^{n{-}1}.
\end{split}
\end{equation}
Obviously $\mathfrak{B}(n)$ is $n$-step nilpotent Lie algebra and real $1$-forms $x^1, y^1, x^2$ are closed 
and $\dim{H^1(\mathfrak{B}(n))}=b_1(\mathfrak{B}(n))=3$.
\end{example}

\begin{remark}
A complex structure $J$ on a nilpotent Lie algebra
$\mathfrak{g}$ corresponding to a basis
$\{\omega^1, \dots \omega^n, \bar \omega^1, \dots \bar \omega^n \}$
of $\mathfrak{g}^*$ such that
$$d\omega^{i+1} \in \Lambda^2(\omega^1, \dots \omega^i,
\bar \omega^1, \dots \bar \omega^i), \; i=0, \dots, n{-}1,$$
was called in \cite{CFGU} a nilpotent complex structure.
\end{remark}
One can 
give an invariant definition of nilpotent complex structure.

\begin{definition}[\cite{Sal}]
An almost complex structure $J$ on a Lie algebra
$\mathfrak{g}$ is called (integrable) nilpotent complex structure
if for all $l=1,\dots, s(\mathfrak{g}),$
$$d V_l^{1,0} \subset V_{l{-}1}^{1,0}\wedge (V_{l{-}1}^{1,0}
\oplus V_{l{-}1}^{0,1}).$$
\end{definition}

\begin{proposition}
A nilpotent $2n$-dimensional Lie algebra $\mathfrak{g}$ admits
a complex Lie algebra structure if and only if
$(\mathfrak{g}^*)^{\mathbb C}$ has a basis
$\{\omega^1, \dots \omega^n, \bar \omega^1, \dots \bar \omega^n \}$
such that
$$d\omega^{i+1} \in \Lambda^2(\omega^1, \dots \omega^i),
\; i=0, \dots, n{-}1.$$
\end{proposition}

Obviously an abelian complex structure and a complex Lie algebra structure are examples
of nilpotent complex structures. However there  are examples of nilpotent Lie algebras that admit only non-nilpotent complex structures.

\begin{example}[S.~Salamon, \cite{Sal}]
\label{ex_{6,8}}
Let us consider a $6$-dimensional Lie algebra $\mathfrak{g}_{6,8}$
defined by its basis $e_1, \dots, e_6$ and structure relations:
$$
[e_1,e_2]=e_3, \:[e_1,e_3]=e_4, \:[e_2,e_3]=e_5, \:[e_1,e_4]=[e_2,e_5]=e_6.
$$
One can verify that an almost complex structure $J$ on $\mathfrak{g}_{6,8}$ defined by 
$$
Je_1=-e_2, Je_2=e_1, \; Je_4=-e_5, Je_5=e_4, \; Je_3=e_6, Je_6=-e_3,
$$
is integrable and non-abelian.

Proceeding to the dual picture (taking the dual basis $e^1, e^2, \dots, e^6$ of $\mathfrak{g}_{6,8}^*$)
we have the following relations for the differentials of basic forms:
$$
de^1=0, \:de^2=0, \: de^3=e^1{\wedge}e^2,
\: de^4=e^1{\wedge}e^3, \:de^5=e^2{\wedge}e^3,
\:de^6=e^1{\wedge}e^4+e^2{\wedge}e^5.
$$

Now let us consider the complexification $(\mathfrak{g}_{6,8}^{\mathbb C})^*$ and complex forms 
$$
\omega^1=e^1+ie^2,\:\omega^2=e^4+ie^5, \:\omega^3=e^6+ie^3.
$$
They are a basis of holomorphic forms $\Lambda^{1,0}(\mathfrak{g}_{6,8}^{\mathbb C})^*$.

Now one can verify the formulas
$$
d\omega^1=0, \; d\omega^2=
\omega^1{\wedge}\frac{i}{2}(\bar \omega^3{-}\omega^3), \:
d\omega^3=\frac{1}{2}(\omega^1{\wedge}\bar \omega^2+
\bar \omega^1 {\wedge} \omega^2)-
\frac{1}{2}\omega^1{\wedge}\bar \omega^1.
$$
\end{example}

One can generalize this example for an arbitrary even dimension $\dim{\mathfrak{g}} \ge 6$.
\begin{example}
A Lie algebra
$\mathfrak{C}(n{+}1), n \ge 3,$:
let $w_2, w_{n+1}, x_{l}, y_{l}, l=1,3,4, {\dots}, n,$ be  its basis.
Structure constants are defined by:
\begin{equation}
\begin{split}
[x_1,y_1]{=}z_2, [x_1,w_2]{=}x_3, [y_1,w_2]{=}y_3, 
[x_1,x_n]{=}[y_1,y_n]{=}w_{n{+}1},\\
[x_1,x_l]{=}[y_1,y_l]{=}x_{l{+}1},
[y_1,x_l]{=}[x_1,y_l]{=}y_{l{+}1},\;
l=3, {\dots}, n{-}1.
\end{split}
\end{equation}
We define a
complex structure $J$:
$$Jw_2=w_{n{+}1}, Jw_{n{+}1}=-w_2, Jx_l=-y_l, \;Jy_{l}=x_l, \; l=1,3,4,\dots,n.$$

Taking $1$-forms $\omega^2=w^{n{+}1}+iw^2, \omega^{l}=x^l+iy^{l}, \; l=1,3,4,\dots,n,$
we have
\begin{equation}
\begin{split}
d\omega^1=0, \; d\omega^2=\omega^1 {\wedge}\frac{i}{2} \left(\bar \omega^{n{+}1}{-}\omega^{n{+}1}\right), 
d\omega^3=\bar \omega^1 {\wedge} \omega^2, \dots,\\
d\omega^{n}=\bar \omega^1 {\wedge} \omega^{n{-}1},\;
d\omega^{n{+}1}=\frac{1}{2}\left(\omega^1 {\wedge} \bar \omega^{n}+\bar \omega^1 {\wedge}  \omega^{n}\right)-
\frac{1}{2} \omega^1\wedge\bar  \omega^1.
\end{split}
\end{equation}
\end{example}

\section{Integrable complex structures and algebraic constraints}

Consider again the decreasing sequence of subalgebras 
$\mathfrak{g}^k(J)$:
\begin{equation}
\label{J-invar}
\mathfrak{g}=\mathfrak{g}^1(J) \supset \mathfrak{g}^2(J) \supset \dots
\supset \mathfrak{g}^{s(\mathfrak{g})}(J) \supset \{0\}.
\end{equation}
We have already noted  that the first inclusion  in this sequence is strict, but this is not necessarily so for the other inclusions.

Recall the example \ref{ex_{6,8}} of the $6$-dimensional nilpotent Lie algebra $\mathfrak{g}_{6,8}$
endowed with the non-abelian complex structure $J$. $\mathfrak{g}_{6,8}$ is $4$-step nilpotent and one can easily remark that
$$
\mathfrak{g}_{6,8} \supset \mathfrak{g}_{6,8}^2(J) = \mathfrak{g}_{6,8}^3(J) 
\supset \mathfrak{g}_{6,8}^4(J) \supset \{0\},
$$
where
$$
\mathfrak{g}_{6,8}^2(J){=} Span (e_3,e_4,e_5,e_6){=} \mathfrak{g}_{6,8}^3(J), \;\;
 \mathfrak{g}_{6,8}^4(J){=}Span (e_5,e_6).
$$
Let $E$ denotes  the total number of equalities in the sequence (\ref{J-invar}). 

Obviously we have the following estimate:
\begin{equation}
\label{E1}
2\left(s(\mathfrak{g})-E \right) \le \dim{\mathfrak{g}},
\end{equation}
it follows from the fact that the dimension $\dim{\mathfrak{g}^{k}(J)}$  decreases at each strict inclusion at least by two.

\begin{proposition}
Let $\mathfrak{g}$ be a Lie algebra endowed by nilpotent complex structure. Then 
we have the following estimate on its nil-index:
\begin{equation}
\label{nilptotent_complex}
s(\mathfrak{g}) \le \frac{1}{2} \dim{\mathfrak{g}}.
\end{equation}
\end{proposition}

\begin{proof}
It follows that for a nilpotent complex structure $J$
in the decreasing sequence (\ref{J-invar}) 
all inclusions are strict, i.e. $E=0$.
The Example \ref{abelian_long} shows that our estimate on nil-index of nilpotent Lie algebras with nilpotent or abelian complex structures is sharp.

The estimate (\ref{nilptotent_complex}) of nil-index have not been discussed in \cite{CFGU}, although
it easily follows from the arguments there.
\end{proof}

\begin{remark}[\cite{CFGU}]
Another restriction imposed on the algebraic structure of $\mathfrak{g}$
by the existence of a nilpotent complex structure was the following estimate \cite{CFGU}:
$$
b_1(\mathfrak{g})=\dim{\mathfrak{g}} - \dim{[\mathfrak{g}, \mathfrak{g}]}\ge 3.
$$
This estimate immediately follows from the properties of
the canonical basis $\omega^1, \bar \omega^1, \dots, \omega^n, \bar \omega^n$
in $(\mathfrak{g}^*)^{\mathbb C}$. Real and imaginary parts of $\omega^1$ are
linearily independent closed forms. If $d\omega^2=0$ 
then it evidently follows that $b_1(\mathfrak{g})\ge 4$. 
If $d\omega^2  \ne 0$ one can choose
$\omega_2$ such that $d\omega^2=\omega^1 \wedge \bar \omega^1$, but
the $2$-form $\omega^1 \wedge \bar \omega^1$ is pure imaginary
and hence the real part of $\omega^2$ must to be closed, that
gives us at least three linearily independent closed real $1$-forms.
\end{remark}

The complex Lie algebra structure $J$ can be regarded as abelian complex structure as we have seen.
But in order to have sharp estimates we have to precise our estimates.
\begin{proposition}
If a real nilpotent Lie algebra $\mathfrak{g}$ admits
a complex Lie algebra structure, then
\begin{equation}
s(\mathfrak{g}) \le \frac{1}{2}\dim{\mathfrak{g}}-1, \;\;\;\;\;
b_1(\mathfrak{g})=\dim{\mathfrak{g}} - \dim{[\mathfrak{g}, \mathfrak{g}]}\ge 4.
\end{equation}
\end{proposition}

\begin{proof} $2n$-dimensional real Lie algebra
$\mathfrak{g}$ can be regarded as a complex $n$-dimensional Lie algebra
and hence 
$$s(\mathfrak{g}) \le n-1=\frac{1}{2}\dim{\mathfrak{g}}-1.$$
Its first cohomology groupe $H^1(\mathfrak{g})$ as a complex space has dimension at least $2$.
Hence
 $\dim{\mathfrak{g}} - \dim{[\mathfrak{g}, \mathfrak{g}]}=b_1(\mathfrak{g}) \ge 4$.
The Lie algebra $\mathfrak{m}_0^{\mathbb R}(n)$
from Example \ref{compl_filif} shows that these estimates are sharp.
\end{proof}

As we have already noticed, the nil-index $s(\mathfrak{g})$ of even-dimesional real nilpotent Lie algebra 
$\mathfrak{g}$
 can exceed  the value $\frac{1}{2}\dim{\mathfrak{g}}$ if we have positive number $E >0$ of equalities in the sequence (\ref{J-invar}).

\begin{lemma}
\label{mainlemma}
Let $k \ge 2$ and
\begin{equation}
\label{jump_cond}
\mathfrak{g}^{k{-}1}(J) \supset \mathfrak{g}^k(J)=\mathfrak{g}^{k{+}1}(J)=\dots=\mathfrak{g}^{k{+}p}(J) \supset
\mathfrak{g}^{k{+}p{+}1}(J), \; p \ge 1,
\end{equation}
where the first and the last inclusions are strict.  
Then  
$$
\dim{\mathfrak{g}^{k{+}p{+}1}} {-} \dim{\mathfrak{g}^{k{+}p{+}2}}\ge 2.
$$
\end{lemma}

\begin{proof} 
The condition (\ref{jump_cond}) is equivalent to the dual one
$$
V_{k{-}1}^{1{,}0} \subset V_k^{1{,}0} =V_{k{+}1}^{1{,}0} =
\dots =V_{k{+}p}^{1{,}0} \subset V_{k{+}p{+}1}^{1{,}0}, \; k \ge 2.
$$

Let us consider $\omega \in V_{k{+}p{+}1}^{1{,}0}, \omega
\notin V_{k{+}p}^{1{,}0}$. Fix a base
$\omega^1_1,\dots, \omega^{j_1}_1$ of $V_1^{1{,}0}$.  Add thereto new elements $\omega^1_{2},\dots,\omega^{j_2}_2$ (if necessary) to obtain a basis of $V_2^{1{,}0}$. Continue this process sequentially.
In the last step we add $1$-forms $\omega^{1}_{k}, \dots, \omega^{j_{k}}_{k} \in  V_{k}^{1{,}0}$ such that 
the whole set 
$$
\omega^1_1, \dots, \omega^{j_1}_1, \omega^1_2, \dots, \omega^{j_2}_2, \dots, \omega^1_k, \dots, \omega^{j_k}_k,
$$
 would form a basis of $V_{k}^{1{,}0}$.
It follows from (\ref{d-J-model}) that 
\begin{equation}
\begin{split}
d\omega=\Omega=\omega^1_1 \wedge \xi^1_1 +\dots +
\omega^{j_1}_1 \wedge \xi^{j_1}_1+\sum_{2 \le m \le k, 1 \le l \le j_m }\omega^l_{m} \wedge \xi^l_{m},\\
\omega^l_m \in V_m^{1,0}, \; \xi^l_m \in V_{k{+}p{-}m} (\mathfrak{g}^*)^{\mathbb C}, \; 1 \le m \le k, \; 1 \le l \le j_m.
\end{split}
\end{equation}

Consider two inclusions
$
\mathfrak{g}^2 \subset \mathfrak{g}^2(J) \subset \mathfrak{g}.
$ 
We recall that the second one is strict. We choose a basis in the annihilator $\mathfrak{g}^2(J)^{ann}$ 
and complete it (if necessary, i.e. $\mathfrak{g}^2 \ne \mathfrak{g}^2(J)$) to a whole basic of $V_1\mathfrak{g}$. 
We denote the elements that we add by
$$e^1_1,\dots,e^r_1.$$
Remark that 
$$
\omega^1_1,\dots, \omega^{j_1}_1, \bar{\omega}^1_1,\dots, \bar{\omega}^{j_1}_1, e_1^1, \dots, e_1^{r},
$$
is a basis of subspace $V_1(\mathfrak{g}^*)^{\mathbb C}$. 
We denote by 
$$
e^1_k,e^2_k,\dots, e^{j_k}_k, \;  2\le k \le s(\mathfrak{g}),
$$
linear independent $1$-forms such that 
$$
V_k \mathfrak{g}= Span(e^1_k,e^2_k,\dots, e^{j_k}_k) \oplus V_{k{-}1} \mathfrak{g}, \; 2 \le k \le s(\mathfrak{g}).
$$ 
It is obvious that
$$
V_k (\mathfrak{g})^{\mathbb C}= Span^{\mathbb C}(e^1_k,e^2_k,\dots, e^{j_k}_k) \oplus
 V_{k{-}1} (\mathfrak{g})^{\mathbb C}, \; 2 \le k \le s(\mathfrak{g}).
$$ 

\begin{proposition}
Cohomology classes  $[\Omega]$ and $[\bar{\Omega}]$ are linearly independent in 
$\ker {\varphi_{k{+}p}} \subset H^2(\Lambda^2(V_{k{+}p}(\mathfrak{g}^*)^{\mathbb C}))$,
where the mapping in the cohomology
$$
\varphi_{k{+}p}: 
H^2(\Lambda^2(V_{k{+}p}(\mathfrak{g}^*)^{\mathbb C})) \to
H^2((\mathfrak{g}^*)^{\mathbb C}),
$$
is induced by the inclusion of exterior d-algebras
$$
\Lambda^*(V_{k{+}p}(\mathfrak{g}^*)^{\mathbb C}) \to
\Lambda^*((\mathfrak{g}^*)^{\mathbb C}).
$$
\end{proposition}

\begin{proof}
Among $\xi^1_1,\dots, \xi^{j_1}_1$ there is at least one $\xi^{k_0}_1$ that belongs to 
$V_{k{+}p}(\mathfrak{g}^*)^{\mathbb C}$ and does not belong to $V_{k{+}p{-}1}(\mathfrak{g}^*)^{\mathbb C}$
(otherwise $\Omega \in \Lambda^2(V_{k{+}p{-}1}(\mathfrak{g}^*)^{\mathbb C})$ and hence
$\omega \in V_{k{+}p}(\mathfrak{g}^*)^{\mathbb C}$ which contradicts to our choice of $\omega$).

We decompose $\xi_{k_0}$ into a linear combination of basis vectors:
$$
\xi^{k_0}_1=\sum_{j=1}^{m_{k{+}p}}A_{j,k{+}p} e_{k{+}p}^j+
\sum_{1 \le r \le m_{q} , 1 \le q \le k{+}p{-}1}A_{r,q} e^r_q +\sum_{t=1}^{j_1} B_t \omega_1^t 
+\sum_{s=1}^{j_1} C_s \bar{\omega}_1^s.
$$
There is $l_0, 1 \le l_0 \le j_m$, such that $A_{l_0,k{+}p} \ne 0$. Hence
$$
\Omega=A_{l_0,k{+}p}\omega^{k_0}_1 \wedge e_{k{+}p}^{l_0} + {\Omega}_1, \;A_{l_0,k{+}p}\ \ne 0.
$$
The key observation is that in the expansion of $\Omega$ there is no term proportional to
$\bar \omega^{k_0}_1 \wedge e_{k{+}p}^{l_0}$.
On the other hand, the expansion of an arbitrary coboundary $d\rho, \rho \in V_{k{+}p}(\mathfrak{g}^*)^{\mathbb C}$ can not contain terms 
$\omega^{k_0}_1 \wedge e_{k{+}p}^{l_0}$ and $\bar \omega^{k_0}_1 \wedge e_{k{+}p}^{l_0}$, because
$$
d \rho \in \Lambda^2(V_{k{+}p{-}1}(\mathfrak{g}^*)^{\mathbb C}), \;  \rho \in V_{k{+}p}(\mathfrak{g}^*)^{\mathbb C}.
$$
Thus the 
cohomology classes  $[\Omega]$ and $[\bar{\Omega}]$ are linearly independent in 
$\ker {\varphi_{k{+}p}}$.
\end{proof}

\begin{example}
Consider again Example \ref{ex_{6,8}}. In the first step we chose $\omega^1, \bar \omega^1$ 
as a basis of $V_1(\mathfrak{g}_{6,8})^{\mathbb C}$. The kernel  of the cohomology map induced
by the inclusion
$$
 \Lambda^2(\omega^1, \bar \omega^1) \to \Lambda^2((\mathfrak{g}_{6,8}^*)^{\mathbb C}),
$$
is one-dimensional and it is spanned by $[\omega_1 \wedge \bar \omega_1]={-}2i[e^1 \wedge e^2]$.
In the second step $V_1^{1,0}=V_2^{1,0}$ and we 
added a new generator $e^3$ in order to kill the kernel $\ker{\varphi_1}$:
$$
de^3=e^1 \wedge e^2=
\frac{i}{2} \omega^1 \wedge \bar \omega^1.
$$

The third step. We have a strict inclusion $V_2^{1,0} \subset V_3^{1,0}$ and and we took $\omega^2$, which together with $\omega^1$  constitutes the basis of the subspace $V_3^{1,0}$. Thus the equality holds
$$
d\omega^2=\Omega_2=\omega^1 \wedge e^3.
$$
The cocycles $\Omega_2=\omega^1 \wedge e^3$ and $\bar \Omega_2=\bar \omega^1 \wedge e^3$ are obviously linearly independent
in $\Lambda^2(\omega^1, \bar \omega^1,e^3)$. An arbitrary
$2$-coboundary has the form $\alpha \omega^1 \wedge \bar \omega^1, \alpha \in {\mathbb C}$.
Hence  cohomology classes $[\Omega_2]$ and $[\bar{\Omega}_2]$ are linearly independent in 
$\ker {\varphi_2}$ in $2$-cohomology and they span this kernel. So we add new generators 
$\omega^2, \bar \omega^2$ and kill $\ker{\varphi_2}$.

The last step. Again we have a strict inclusion $V_3^{1,0} \subset V_4^{1,0}$. We add 
$\omega_3$ and get a basis $\omega^1, \omega^2, \omega^3$ of $\Lambda^{1,0}$:
$$
d \omega^3=\Omega_3=\frac{1}{2}\left(\omega^1 \wedge \bar \omega^2 + \bar \omega^1 \wedge \omega^2 \right)-
\frac{1}{2}\omega^1 \wedge \bar \omega^1.
$$
The kernel  of the cohomology map induced
by the inclusion
$$
 \Lambda^2(\omega^1, \bar \omega^1, e^3, \omega^2, \bar \omega^2) \to \Lambda^2((\mathfrak{g}_{6,8}^*)^{\mathbb C}),
$$
is one-dimensional and it is spanned by 
$[\Omega_3]=[\omega^1{\wedge} \bar \omega^2{+} \bar \omega^1 {\wedge} \omega^2]=2[e^1 \wedge e^4 + e^2 \wedge e^5]$.

We see that 
$$
\Omega_3 - \bar \Omega_3= \omega^1 \wedge \bar \omega^1=d(-2ie^3).
$$
So in this case $2$-classes $\Omega_3$ and $\bar \Omega_3$ are linearly dependent in $\ker{\varphi_3}$. There is no contradiction
with Lemma \ref{mainlemma}, because we have the case of two consecutive  strict inclusions:
$$
V_2^{1,0} \subset V_3^{1,0} \subset V_4^{1,0}.
$$
\end{example}

\end{proof}
\begin{proposition}
\label{ge5}
Let $\mathfrak{g}$ be a nilpotent Lie algebra endowed with an integrable complex structure and $\dim{\mathfrak{g}} \ge 6$. Then we have
the following estimate:
\begin{equation}
\label{no_filiform}
\dim{\mathfrak{g}} - \dim{\mathfrak{g}^4}=\dim{\mathfrak{g}} - \dim{\left[\mathfrak{g}, \left[\mathfrak{g}, [\mathfrak{g},\mathfrak{g}]\right]\right]}
\ge 5.
\end{equation}
\end{proposition}

\begin{proof}
Our assertion is necessary to prove only for  Lie algebras $\mathfrak{g}$ with the first Betti number
$b_1(\mathfrak{g})=2$ (for all other algebras our inequality holds automatically). In this case forms $\omega^1, \bar \omega^1$ span the space
$V_1(\mathfrak{g}^*)^{\mathbb C}$ and 
$$
d \omega^2=\omega^1 \wedge \xi^{(0)}, \xi^{(0)} \in V_k(\mathfrak{g}^*)^{\mathbb C}, k \ge 2, 
\xi^{(0)} \notin V_{k{-}1}(\mathfrak{g}^*)^{\mathbb C}.
$$
If $\xi^{(0)} \in V_1(\mathfrak{g}^*)^{\mathbb C}$ then either $\omega^2$ is closed or $\omega^1 \wedge \xi^{(0)}$ is proportional to
$\omega^1 \wedge \bar \omega^1$ which is pure imaginary form and hence in this case
the real part of $\omega^2$ have to be closed and it condradicts to the assumption that
$b_1(\mathfrak{g})=2$.

We see that $d \xi^{(0)}=\omega^1 \wedge \xi^{(1)}, \xi^{(1)} \in V_{k{-}1}(\mathfrak{g}^*)^{\mathbb C}$.

Continuing step-by-step process, we get at the end two linearly independent $2$-cocycles
$$
\omega^1 \wedge \xi^{(k{-}2)}=\omega^1 \wedge (\alpha e^3{+}\beta\omega^1{+}\gamma \bar \omega^1),\;\;
 \bar \omega^1 \wedge \bar \xi^{(k{-}2)}= \bar \omega^1 \wedge (\bar \alpha e^3{+}\bar \beta \bar \omega^1{+}\bar \gamma \omega^1),
$$
where $de^3=e^1{\wedge} e^2, \; \omega^1=e^1{+}ie^2$ and $\alpha \ne 0$.
Hence 
$$
\dim V_3(\mathfrak{g}^*)^{\mathbb C} - \dim V_2(\mathfrak{g}^*)^{\mathbb C}\ge 2,
$$
which means that 
$
\dim{\mathfrak{g}^3} - \dim{\mathfrak{g}^4} \ge 2,
$
Finally we have
$$
\left( \dim{\mathfrak{g}} - \dim{\mathfrak{g}^2}\right) + \left( \dim{\mathfrak{g}^2} - \dim{\mathfrak{g}^3} \right)+
\left(\dim{\mathfrak{g}^3} - \dim{\mathfrak{g}^4}\right) \ge 5.
$$
\end{proof}

As an elementary corollary we get the assertion that was proved in \cite{GzR}:
\begin{corollary}
A filiform Lie algebra $\mathfrak{g}$ does not admit any integrable complex structure.
\end{corollary}
\begin{proof}
It follows from the definition  that for an arbitrary filiform Lie algebra
$\mathfrak{g}$ the following equality holds on:
$$
\dim{\mathfrak{g}}-\dim{\mathfrak{g}^4}=2+1+1=4.
$$
\end{proof}




\section{Main example}

\begin{example}
Define a positively graded
Lie algebra
$\mathfrak{D}(n)=\oplus_{l=1}^{n} \mathfrak{D}_{l}$ such that:
$$\dim \mathfrak{D}_{l}(n)= \left\{\begin{array}{r}
   1, \; \; \quad l=2k, l \le n ;\\
   2, \; l=2k{-}1, l \le n.\\
   \end{array} \right.$$
Let $v_{2k{-}1}, u_{2k{-}1}$ denote basic elements in
$\mathfrak{D}_{2k{-}1}(n)$ and $w_{2k}$ in
$\mathfrak{D}_{2k}(n)$ respectively. Then
a Lie algebra structure of $\mathfrak{D}(n)$ is defined by:
\begin{equation}
\label{slt}
[v_i, w_j ]=\left\{\begin{array}{r} u_{i{+}j}, i{+}j \le n;\\
0, \;\; i{+}j > n.\\
\end{array}
\right., 
[w_j, u_l ]=\left\{\begin{array}{r} v_{j{+}l},  j{+}l \le n;\\
0, \;\;  j{+}l >n.\\
\end{array} \right.,
[u_l, v_i ]=\left\{\begin{array}{r} w_{l{+}i}, l{+}i \le n;\\
0, \;\; l{+}i > n.\\
\end{array}
\right.
\end{equation}
We recall that indexes $i, l$ (index $j$) in (\ref{slt}) are taking odd
(even) positive integer values.
\end{example}

\begin{remark}
$\mathfrak{D}(4)$ is isomorphic to the algebra $\mathfrak{g}_{6,8}$ from the Example \ref{ex_{6,8}}.
\end{remark}

\begin{proposition}
$\mathfrak{D}(n)$ is naturally graded nilpotent Lie algebra and 
$$
\dim{\mathfrak{D}(n)}=\left\{ \begin{array}{ll}6m,& n=4m,\\
6m+2,& n=4m{+}1,\\ 6m+3,& n=4m{+}2,\\ 6m+5,& n=4m{+}3,\\ 
\end{array}\right., \;  s(\mathfrak{D}(n))=n.
$$
\end{proposition}

\begin{proof}
It is easy to see that
$$
\mathfrak{D}(n)^2=\left[\mathfrak{D}(n), \mathfrak{D}(n) \right]=\oplus_{l=2}^{n} \mathfrak{D}_{l}.
$$
Continue this process step-by-step we have that
$$
\mathfrak{D}(n)^m=\left[\mathfrak{D}(n), \mathfrak{D}(n)^{m{-}1} \right]=\oplus_{l=m}^{n} \mathfrak{D}_{l}, \; m=2,\dots,n.
$$
Hence $\mathfrak{D}(n)^m/\mathfrak{D}(n)^{m{+}1}=\mathfrak{D}_m(n)$ that proves that
$\mathfrak{D}(n)$ is naturally graded. The formulas for the dimension and nil-index can be easily verified.
\end{proof}

\begin{proposition}
The positively graded (nilpotent) Lie algebras $\mathfrak{D}(4m)$ and $\mathfrak{D}(4m{+}1)$ admit
complex structures.
\end{proposition}
\begin{proof}
Define an almost complex structure $J$ on $\mathfrak{D}(4m)$ by
\begin{equation}
\label{J-main}
Jv_{2l{+}1}=u_{2l{+}1},l=0,1,\dots, 2m{-}1, \;
Jw_{4k+2}=w_{4k+4}, \; k=0,1,\dots, m{-}1.
\end{equation}
The proof consists
of verifying the integrability condition (\ref{Nijenhuis}) for basic elements
$u_j, v_k, w_m$.

To define a complex structure on $\mathfrak{D}(4m{+}1)$ one should add to (\ref{J-main})
one relation:
\begin{equation}
\label{J_4m+1}
Jv_{4m{+}1}=u_{4m{+}1}, \; Ju_{4m{+}1}=-v_{4m{+}1}.
\end{equation}

Taking the duals $u^r, v^p, w^l$
one can define the following $1$-forms for $k=1,2,\dots, m{-}1$:
\begin{equation}
\begin{split}
\omega_{3k}=w^{4k{-}2}-iw^{4k}, \;\;\\
\omega_{3k{+}1}=u^{4k{+}1}+iv^{4k{+}1}, \\
\omega_{3k{+}2}=u^{4k{+}3}+iv^{4k{+}3},
\end{split}
\end{equation}
It follows from (\ref{slt}) that
\begin{equation}
du^r=\sum_{i{+}j=r}v^i{\wedge}w^j,\;
dv^p=\sum_{i{+}j=p}w^i{\wedge}u^j,\;
dw^l=\sum_{i{+}j=l}u^i{\wedge}v^j.
\end{equation}
One can easily write out the formulae for the
differentials $d\omega_i$:
\begin{equation}
\begin{split}
d\omega_1=0, \; d\omega_2=-\omega_1 {\wedge} iw^2, \;
d\omega_3=\frac{1}{2}(\omega_1 {\wedge} \bar \omega_2 -
 \bar \omega_1 {\wedge} \omega_2)
+ \frac{i}{2}\omega_1 {\wedge} \bar \omega_1,\\
d\omega_4=-\omega_1 {\wedge} iw^4
- \omega_2 {\wedge} iw^2, \;
d\omega_5=-\omega_1 {\wedge} iw^6
- \omega_2 {\wedge} i w^4 - \omega_4 {\wedge} iw^2,\\
d\omega_6=\frac{1}{2}(\omega_1 {\wedge} \bar \omega_5 {-}
 \bar \omega_1 {\wedge} \omega_5)+\frac{1}{2}
(\omega_2 {\wedge} \bar \omega_4 {-}  \bar \omega_2 {\wedge} \omega_4)+
\frac{i}{2}(\omega_1 {\wedge} \bar \omega_4 {-}
\bar \omega_1 {\wedge} \omega_4)
+ \frac{i}{2}\omega_2 {\wedge} \bar \omega_2,\\
\dots, \quad\quad\quad\quad\quad\quad\quad\quad\quad\quad\quad \dots
\end{split}
\end{equation}
\end{proof}

Remark that the algebras $\mathfrak{D}(4m{+}2)$ and $\mathfrak{D}(4m{+}3)$ are odd-dimensional ones.

\begin{proposition}
Graded Lie algebras $\mathfrak{D}(4m{+}2 \oplus {\mathbb R}$ and 
$\mathfrak{D}(4m{+}3 \oplus {\mathbb R}$ are $(6m{+}4)$-dimensional and $(6m{+}6)$-dimensional respectively and they admit complex structures defined by (\ref{J-main}) and (\ref{J_4m+1}) with
the additional relation
$$
Jw_{4k{+}2}=t, \; Jt=-w_{4k{+}2}.
$$
Obviously $s(\mathfrak{D}(4m{+}2)\oplus {\mathbb R})=4m{+}2$ and 
$s(\mathfrak{D}(4m{+}3)\oplus {\mathbb R})=4m{+}3$ because the new generator $t$ belongs to the centre of our algebra.
\end{proposition}
We proved the Theorem
\begin{theorem}
\label{main_theorem}
Let $s(2n)$ denotes the maximal value of nil-index $s(\mathfrak{g})$ of $2n$-dimensional nilpotent Lie algebra 
$\mathfrak{g}$ that admits a complex structure. Then we have the following estimates:
\begin{equation}
\label{main_estimate}
\left[\frac{4n}{3} \right] \le s(2n) \le 2n-2.
\end{equation}
\end{theorem}
\begin{remark}
It follows from \cite{Sal} that $s(6)=4$. But it appears possible to improve the estimates (\ref{main_estimate}) for dimensions $2n \ge 8$.
\end{remark}

\bibliographystyle{amsalpha}

\begin{thebibliography}{A}


\bibitem{BD1}
M.~Barberis, I.~Dotti,
{\it Abelian complex structures on solvable Lie algebras},
J. of Lie Theory, {\bf 14}:1 (2004), 25--34.

\bibitem{BD2}
M.~Barberis, I.~Dotti,
{\it Complex structures on affine motion groups},
Quart. J. Math., {\bf 55}:4 (2004), 375--389.

\bibitem{BG}
C.~Benson, C.~Gordon,
{\it K\"ahler and symplectic structures on nilmanifolds},
Topology, {\bf 27}:4 (1988), 513--518.

\bibitem{CF}
S.~Console, A.~Fino,
{\it Dolbeault cohomology of compact nilmanifolds},
Transform. Groups, {\bf 6} (2001), 111--124.

\bibitem{CFGU}
L.A.~Cordero, M.~Fernandez, A.~Gray, L.~Ugarte,
{\it Nilpotent complex structures on compact nilmanifolds},
Rend.Circolo Mat.Palermo {\bf 49} suppl. (1997), 83--100.

\bibitem{DF1}
I.~Dotti, A.~Fino,
{\it Hypercomplex eight-dimensional nilpotent Lie groups},
J. of Pure and Appl. Algebra {\bf 184} (2003), 47--57.

\bibitem{DF2}
I.~Dotti, A.~Fino,
{\it Hypercomplex nilpotent Lie groups},
in "Global Differential geometry: The Mathematical Legacy
of Alfred Gray (Bilbao, 2000)", 310--314, Contemp. Math., {\bf 288},
AMS, Providence, RI, 2001.

\bibitem{GzR}
M.~Goze, E.~Remm,
{\it Non existence of complex structures on filiform Lie algebras},
Commun. Algebra {\bf 30}:8 (2002), 3777-3788.


\bibitem{Has}
K. Hasegawa,
{\it Minimal models of nilmanifolds},
Proc. Amer. Math. Soc. {\bf 106}:1 (1989), 65--71.



\bibitem{Moros}
V.~Morosov, {\it Classification of nilpotent Lie algebras of order $6$}, Izv. Vyssh. Uchebn. Zaved. Mat.
{\bf 4} (1958), 161–171.

\bibitem{NN}
A.~Newlander, L.~Nirenberg,
{\it Complex analytic coordinates in almost complex manifolds},
Annals Math. {\bf 65} (1957), 391--404.


\bibitem{Sal}
S.M.~Salamon,
{\it Complex Structures on Nilpotent Lie Algebras},
J.Pure Appl.Algebra, {\bf 157} (2001), 311--333.

\bibitem{Sal2}
S.M.~Salamon,
{\it Almost parallel structures},
in "Global Differential geometry: The Mathematical Legacy
of Alfred Gray (Bilbao, 2000)", 162--181, Contemp. Math., {\bf 288},
AMS, Providence, RI, 2001.

\bibitem{V}
M.~Vergne,
{\it Cohomologie des alg\`ebres de Lie nilpotentes},
Bull. Soc. Math. France {\bf 98} (1970), 81--116.

\end{thebibliography}

\end{document}